\documentclass[11pt,letterpaper]{amsart}

\usepackage[dvipsnames]{xcolor}
\usepackage{tikz}
\usepackage{stmaryrd}
\usetikzlibrary{matrix,arrows,decorations.pathmorphing}
\usepackage{graphicx}
\usepackage{setspace}
\usepackage{amssymb}
\usepackage[stable]{footmisc}
\usepackage{amsmath,mathtools}
\usepackage{fullpage}
\usepackage{caption}
\usepackage{amsthm}
\usepackage{mathrsfs}
\usepackage{thmtools}
\usepackage{blkarray}
\usepackage{multirow}
\usepackage{picinpar} 
\usepackage{tikz-cd}
\usepackage{color}
\usepackage{verbatim}
\usepackage{hyperref}
\hypersetup{colorlinks=true, linkcolor=black, citecolor = OliveGreen, urlcolor = magenta}
\usepackage{amssymb}
\usepackage{amsmath}
\usepackage{enumitem,colonequals}
\usepackage{color}
\usepackage{mathrsfs}
\usepackage[all]{xy}
\usepackage{comment}

\usepackage[OT2,T1]{fontenc}
\DeclareSymbolFont{cyrletters}{OT2}{wncyr}{m}{n}
\DeclareMathSymbol{\Sha}{\mathalpha}{cyrletters}{"58}

\DeclareFontFamily{OT1}{rsfs}{}
\DeclareFontShape{OT1}{rsfs}{n}{it}{<-> rsfs10}{}
\DeclareMathAlphabet{\mathscr}{OT1}{rsfs}{n}{it}

\usepackage{relsize}
\usepackage[bbgreekl]{mathbbol}

\usepackage{tikz}
\usetikzlibrary{arrows,matrix}
\usepackage[all]{xy}

\usepackage{hyperref}
\hypersetup{
    colorlinks=true, 
    linkcolor=black,
    filecolor=magenta,
    urlcolor=magenta,
    citecolor=cyan}

\RequirePackage{mathrsfs} 

\theoremstyle{theorem}
\numberwithin{equation}{subsection}
\newtheorem{thmx}{\text{Theorem}}

\newtheorem{theorem}[subsubsection]{Theorem}

\newtheorem{lem}[subsubsection]{Lemma}
\newtheorem{lemma}[subsubsection]{Lemma}
\newtheorem{cor}[subsubsection]{Corollary}
\newtheorem{prop}[subsubsection]{Proposition}

\numberwithin{equation}{subsection}

\theoremstyle{definition}

\newtheorem{remark}[subsubsection]{Remark}

\theoremstyle{remark}

\newcommand{\ZZ}{\mathbb Z}

\newcommand{\mZ}{\mathbb{Z}}

\newcommand{\mQ}{\mathbb{Q}}
\newcommand{\mF}{\mathbb{F}}

\newcommand{\mP}{\mathbf{P}}

\renewcommand{\tilde}{\widetilde}

\newcommand{\sE}{{\mathscr E}}
\newcommand{\sF}{{\mathscr F}}

\newcommand{\sH}{{\mathscr H}}

\newcommand{\sL}{{\mathscr L}}

\newcommand{\sN}{{\mathscr N}}
\newcommand{\sO}{{\mathscr O}}

\newcommand{\id}{\operatorname{id}}

\newcommand{\cdef}[1]{\textsf{\textit{#1}}}

\newcommand{\Sym}{\operatorname {Sym}}

\newcommand{\Hom}{\operatorname{Hom}}

\def\Spec{\operatorname{Spec}}

\def\cchar{\operatorname{char}}

\def\int{\operatorname{int}}

\def\div{\operatorname{div}}

\newcommand{\PP}{\mathbf P}

\def\cX{\mathcal{X}}
\def\cY{\mathcal{Y}}

\newcommand{\Arg}{\rule{1ex}{1pt}}

\newcommand{\ux}{\underline{x}}

\newcommand{\Spf}{\operatorname{Spf}}

\DeclareMathOperator{\alg}{alg}

\DeclareMathOperator{\AffProl}{\mathsf{AffProl}}
\DeclareMathOperator{\Prol}{\mathsf{Prol}}

\def\ux{\underline x}

\DeclareMathOperator{\tors}{tors}
\DeclareMathOperator{\CH}{CH}
\DeclareMathOperator{\FrSch}{FrSch}
\DeclareMathOperator{\RSpec}{\mathbf{Spec}}

\def\Prol{\mathsf{Prol}}

\DeclareMathOperator{\FT}{FT}

\DeclareSymbolFontAlphabet{\mathbb}{AMSb} 
\DeclareSymbolFontAlphabet{\mathbbl}{bbold}

\colorlet{DG}{green!50!black}
\colorlet{DO}{orange!50!black}
\colorlet{DR}{red!50!black}
\colorlet{DB}{blue!50!black}
\colorlet{DY}{yellow!50!black}

\newcommand{\nptors}{\operatorname{non-}p\operatorname{-tors}}

\def\bG{\mathbf{G}}

\def\div{\textrm{div}}

\begin{document}
\title{Higher dimensional geometry of $p$-jets}

\author{Lance Edward Miller}
\address{Lance Edward Miller \\
	Department of Mathematical Sciences\\
	University of Arkansas \\
	Fayetteville, AR 72701, USA}
\email{lem016@uark.edu}

\author{Jackson S. Morrow}
\address{Jackson S. Morrow \\
	Department of Mathematics\\
	University of North Texas \\
	Denton, TX 76203, USA}
\email{jackson.morrow@unt.edu}

\begin{abstract}
In this work, we prove a quantitative version of the prime-to-$p$ Manin--Mumford conjecture for varieties with ample cotangent bundle. 

 More precisely, let $A$ be an abelian variety defined over a number field $F$, and let $X$ be a smooth projective subvariety of $A$ with ample cotangent bundle. 
We prove that for every prime $p\gg 0$, the intersection of $X(F^{\text{alg}})$ and the geometric prime-to-$p$ torsion of $A$ is finite and explicitly bounded by a summation involving cycle classes in the Chow ring of the reduction of $X$ modulo $p$.  
This result is a higher dimensional analogue of Buium's quantitative Manin--Mumford for curves.  

Our proof follows a similar outline to Buium's in that it heavily relies on his theory of arithmetic jet spaces. 
In this context, we prove that the special fiber of the arithmetic jet space associated to a model of $X$ is affine as a scheme over $\mathbb{F}_p^{\text{alg}}$.  As an application of our results,  we use a result of Debarre to prove that when $X$ is $\mathbb{Q}^{\text{alg}}$-isomorphic to a complete intersection of $c > \dim(A)/2$ many general hypersurfaces of $A_{\mathbb{Q}^{\text{alg}}}$ of sufficiently large degree, the intersection of $X(F^{\text{alg}})$ and the geometric prime-to-$p$ torsion of $A$ is bounded by a polynomial that depends only on $p$, the dimension of the ambient abelian variety, and intersection numbers of certain products of the hypersurfaces. 
\end{abstract}

\date{\today}

\subjclass
{\href{https://mathscinet.ams.org/mathscinet/msc/msc2020.html?t=&s=14G05&btn=Search&ls=s}{14G05} 
(\href{https://mathscinet.ams.org/mathscinet/msc/msc2020.html?t=14G05&s=14G20&btn=Search&ls=Ctt}{14G20},  
\href{https://mathscinet.ams.org/mathscinet/msc/msc2020.html?t=14G20&s=12H25&btn=Search&ls=Ct}{12H25},  
\href{https://mathscinet.ams.org/mathscinet/msc/msc2020.html?t=12H25&s=14G45&btn=Search&ls=Ct}{14G45})} 

\maketitle

\section{Introduction}

The goal of this work is to prove the following theorem.

\begin{thmx}\label{thm:explicit_unramified_MM}
Let $F$ be a number field, $A$ an abelian $F$-variety of dimension $n$, and $X$ a smooth $F$-subvariety of dimension $d$ in $A$ such that the cotangent bundle $\Omega_X^1$ of $X$ is ample. For every $p\gg 0$ of good reduction for both $X$ and $A$, 
\[
\# (X(F^{\alg}) \cap A(F^{\alg})[\nptors]) \leq p^{3n} 3^{n} n!  \left( \sum_{i=0}^d  {2d \choose d + i} \deg((-1)^is_i(F^{*}_{X_0}\Omega_{X_0}^1)\cdot \sO_{X_0}(3\Theta_0)^{d-i})\right)
\]
where 
$A(F^{\alg})[\nptors]$ denotes the prime-to-$p$ torsion in $A(F^{\alg})$, $X_0$ and $A_0$ denote the reduction of $X$ and $A$ modulo $p$, respectively, $(-1)^is_i(F^{*}_{X_0}\Omega_{X_0}^1)$ is the $i$-th Segre class of the Frobenius pullback of the cotangent bundle of $X_0$, and $\Theta_0$ is a $\Theta$-divisor on the abelian variety $A_0$. 
\end{thmx}

\autoref{thm:explicit_unramified_MM} can be seen as a higher dimensional generalization of the quantitative Manin--Mumford theorem for curves. We remark that the amplitude of the cotangent bundle is one of many `hyperbolicity conditions', i.e., conditions that generalize genus $g \geq 2$ for curves.

While the Manin--Mumford conjecture was originally established by Raynaud \cite{Raynaud:ManinMumford}, we follow the method of Buium \cite{Buium:GeometrypJets}. This method utilized his theory of arithmetic jet spaces and cemented their role as a tool for producing finiteness theorems in Diophantine geometry.  
Buium's work combined a result of Coleman \cite{Col87} and a calculation over local fields which can be made explicit; we refer the reader to Section \ref{sec:MMProof} for further discussion. 
Note that this effective method inspired algorithmic improvements \cite{Poonen:2001:CTP}. 
As we follow Buium's approach, we also give an effective version.

Our proof of \autoref{thm:explicit_unramified_MM} follows from a local unramified Manin--Mumford statement, which is similar to \cite[Theorem~1.11]{Buium:GeometrypJets}.  
To state our result,  fix $p$ a prime integer,  and let $R = W(\mF_p^{\alg})$ be the completed ring of integers of the maximal unramified extension of $\mQ_p$ and $K = R[1/p]$. 
The explicit local form requires the existence of a curve of genus $g \geq 2$ in the special fiber with a specific relationship between $p$, $g$, and $\dim(X)$. 

\begin{thmx}[Quantitative, unramified Manin--Mumford for varieties with ample cotangent bundle modulo $p\gg 0$]\label{thm:mainthm1}
Let $X$ be a smooth projective $R$-subvariety of dimension $d$ of an abelian $R$-scheme $A$ of dimension $n$ such that the special fiber $X_0$ has ample cotangent bundle. 
Suppose that there exists a smooth projective $R$-curve $C$ of genus $g \geq 2$ with a closed embedding $\iota\colon C \hookrightarrow X$. When $p>2\dim(X)^2g$,
\[
\#(X(K) \cap A(K)[\tors]) \leq p^{3n} 3^{n} n!  \left( \sum_{i=0}^d  {2d \choose d + i} \deg((-1)^is_i(F^{*}_{X_0}\Omega_{X_0}^1)\cdot \sO_{X_0}(3\Theta_0)^{d-i})\right)
\]
where $(-1)^is_i(F^{*}_{X_0}\Omega_{X_0}^1)$ is the $i$-th Segre class of the Frobenius pullback of the cotangent bundle of $X_0$ and $\Theta_0$ is a $\Theta$-divisor on the abelian variety $A_0$. 
\end{thmx}

When $X$ is a curve, the bound from \autoref{thm:mainthm1} specializes to Buium's bound \cite{Buium:GeometrypJets}, with minor improvement (cf.~Subsection \ref{subsec:finite_unramified} and \autoref{rem:Buium_bound}). 
Also, in \autoref{lemma:findingprime}, we show that for $X$ a smooth $F$-subvariety of an abelian $F$-variety with ample cotangent bundle and all $p\gg 0$, the base change of $X$ to $K$ will admit an $R$-model where the conditions of \autoref{thm:mainthm1} are satisfied. 
Thus \autoref{thm:explicit_unramified_MM} will follow from \autoref{thm:mainthm1}.

Our proof of \autoref{thm:mainthm1} requires generalizing two steps of Buium's proof.  
The first involves the geometry of arithmetic jet spaces $J^nX$ introduced in \cite{Buium:GeometrypJets}. Here, the critical point is to use the technical conditions on the special fiber of the smooth $R$-variety $X$ to ensure the special fiber of the first arithmetic jet space $J^1X$ is affine, a result we record as the generalization of this step of Buium's proof required a novel observation. 

\begin{thmx}[Affineness of special fibers of arithmetic jet spaces]
\label{thm:main_Affine}
Let $X$ be a smooth projective $R$-variety such that the special fiber $X_0$ has ample cotangent bundle. Suppose there exists a smooth projective $R$-curve $C$ of genus $g \geq 2$ with a closed embedding $\iota\colon C \hookrightarrow X$ and that $p>2\dim(X)^2g$. For every $n\geq 1$, the special fiber $J^n(X_0)$ of the arithmetic $n$-jet space $J^nX$ is affine. 
\end{thmx}

The final generalization concerns the effective bounds from Buium's work \cite{Buium:GeometrypJets}. 
Buium's effectiveness arises from an explicit intersection theory calculation. 
For curves, this results in a simplified expression involving a prime of good reduction $p$ and the genus of the curve. 
In higher dimensions, this intersection calculation naturally takes place in the Chow ring of the special fiber, hence its expression in terms of Segre classes on the special fiber.

\subsection*{Applications}
We provide two applications of our main theorems.  The first application is the full Manin--Mumford statement for varieties with ample cotangent bundle. This simpler but more general, ineffective statement, is direct combination of a result of Bogomolov \cite{Bogomolov:LAdicReps} and \autoref{thm:explicit_unramified_MM}, which is logically independent from other known proofs.

\begin{thmx}\label{thm:general_MM}
Let $F$ be a number field and $A$ an abelian $F$-variety. If $X$ is a smooth $F$-subvariety of $A$ such that the cotangent bundle $\Omega_X^1$ of $X$ is ample, then
\[
\# (X(F^{\alg}) \cap A(F^{\alg})[\tors]) < \infty. 
\]
\end{thmx}

The final application is two-fold.  
First, we identify a natural explicit class of varieties where the conditions of \autoref{thm:mainthm1} are satisfied. 
This class is formed as a complete intersection of sufficiently ample general hypersurfaces in an abelian $\mQ^{\alg}$-variety. We use results of Debarre \cite{Debarre:VarietiesAmpleCotangent, Debarre:CorrigendumVarietiesAmple} and \autoref{lemma:findingprime} to deduce that they satisfy \autoref{thm:mainthm1}. 
For the second aspect, we use a result of Scarponi \cite{Scarponi:Sparsity}. This work studies the first critical scheme and relates this to torsion. In the case this critical scheme is finite, \cite[Theorem 6.1]{Scarponi:Sparsity} witnesses the same bound given in \autoref{thm:mainthm1}. In the case of general complete intersections in an abelian variety, Scarponi gives a detailed study of the intersections of Segre classes appearing in this bound. Thus, as \autoref{thm:mainthm1}, which is logically independent from \cite{Scarponi:Sparsity}, gives the needed finiteness, we are free to utilize this analysis. This allows us to show that the bound from \autoref{thm:mainthm1}, which {\it a priori} involves intersections of Segre classes can be realized in more explicit terms in this case. 

\begin{thmx}\label{thm:intro_complete_intersection_bound}
Let $A$ be an abelian $\mQ$-variety of dimension $n$, and $X$ a smooth $\mQ$-subvariety of $A$ that is $\mQ^{\alg}$-isomorphic to the intersection of $c >n/2$ sufficiently ample general hypersurfaces $H_1,\dots, H_c $ of large and divisible enough degrees $d_1,\dots,d_c$ in $A_{\mQ^{\alg}}$. 
For every sufficiently large prime $p$, 
there exists some constant $\alpha(p,n,\underline{I})$ that is polynomial in $p,n$, and intersection numbers $\underline{I}$ of certain products of the hypersurfaces such that 
\[
\# (X(\mQ^{\alg}) \cap A(\mQ^{\alg})[\nptors]) \leq \alpha(p,n,\underline{I}).
\]
\end{thmx}

\subsection*{Related results}
The Manin--Mumford conjecture has a rich history and many proofs. To provide some context for our results we summarize the relevant literature. For a general overview of this conjecture in the curve setting, we refer the reader to \cite{Tzermias:MMSurvey}. 
Since our main theorems concern quantitative versions of Manin--Mumford, we primarily focus our discussion on works that prove a finiteness result for $X(F^{\alg}) \cap A(F^{\alg})[\tors]$ and whose proof methods provide a bound for this intersection. 
Below, we refer to the intersection $X(F^{\alg}) \cap A(F^{\alg})[\tors]$ as the \cdef{torsion cosets} of $X$.

As mentioned above, our proof technique follows from Buium's proof \cite{Buium:GeometrypJets} of a quantitative Manin--Mumford for curves. 
His bound on torsion cosets of a curve of genus $g\geq 2$ is \textit{explicit} and depends only on $g$ and a prime $p$ of good reduction. 
Recent breakthrough works Dimitrov--Gao--Habegger and K\"uhne \cite{DGH:Uniform, Kuhne:Equidistribution} on Mazur's uniformity conjecture for curves prove a bound on the torsion cosets of the curve depending only on the genus of the curve. 
While this bound has no dependency on a prime of good reduction, it is {\it non-explicit} and it is unclear if the arguments can be pushed to make it explicit without significant new techniques. 

In the higher dimensional setting, we mention a four works. 
Using model theoretic techniques, Hrushovski \cite{Hrushovski:MM} proved an explicit Manin--Mumford in full generality. 
His bound on the torsion cosets of a subvariety $X$ of an abelian variety has the form $\alpha (\deg_{\sL} X)^{\beta}$ where $\sL$ is some ample line bundle on the ambient abelian variety and $\alpha$ and $\beta$ are constants that depend on the dimension of the abelian variety and a prime of good reduction. 
In \cite{DavidPhilippon:Minorations}, David--Philippon proved explicit bounds for the size of the torsion cosets of a subvariety contained in the product of an elliptic curve. Their bound has a similar form to Hrushovski's without the dependence on a prime of good reduction. 
Finally, we mention the work Gao--Ge--K\"uhne \cite{GaoGeKuhne:UniformML} which generalized the results of \cite{DGH:Uniform, Kuhne:Equidistribution} to higher dimensions, and as a by-product produces a bound depending only on $\dim(A) $ and $\deg_{\sL}(X)$ on the torsion coset a subvariety $X$ of an abelian variety. 
As with the curve setting, the structure of this constant is unclear. 

Finally, we comment on a paper of Scarponi \cite{Scarponi:Sparsity}. 
In this work, the author proves a higher dimensional version of a local result of Raynaud \cite{Raynaud:ManinMumford}, which was used in his proof of the Manin--Mumford conjecture.  
In particular, the author shows that for a subvariety $X$ of an abelian variety over $F$ with trivial stabilizer and $p\gg 0$, the collection of $p$-divisible unramified liftings is not Zariski dense. 
By further studying the Greenberg transformation of level 1 and using interpretations of projective bundles in the spirit of Buium's proof, the author gives an explicit upper bound on the number of irreducible components of the first critical scheme when $X$ is a complete intersection of sufficiently general hypersurfaces in an abelian variety. When this first critical scheme is finite, this bound produces an explicit bound on the prime-to-$p$ torsion cosets of $X$. Our \autoref{thm:mainthm1} provides this finiteness and so we utilize this version of the bound in our proof of \autoref{thm:intro_complete_intersection_bound}. 

\subsection*{Leitfaden}
In Section \ref{sec:prelims},  we establish conventions to be used throughout the work. 
Additionally, we recall background on vector bundles, the geometry of projective bundles, intersection theory, and arithmetic jet spaces. 
The proof of our main results occupy the remainder of the work. 
In Section \ref{sec:affinenes}, we prove our main theorem in the context of arithmetic jet spaces, \autoref{thm:main_Affine}. 
We continue in Section \ref{sec:explicitMM} with a proof of \autoref{thm:mainthm1}. 
The proof breaks down into two parts, one deducing finiteness and the other showing an explicit bound. 
The final two sections are devoted to our above applications. 
In Section \ref{sec:MMProof}, we prove \autoref{thm:explicit_unramified_MM}, which is a quantitative, unramified Manin--Mumford statement, and \autoref{thm:general_MM}, which is a general, yet ineffective, Manin--Mumford statement. This is of note as its logically independent from the other general proofs and follows quite quickly from a result of Bogomolov and Theorem A. Finally, we conclude in Section \ref{sec:computation} with the proof of \autoref{thm:intro_complete_intersection_bound}.

\subsection*{Acknowledgements} 
During preparation of this work, we benefited from conversations with Alexandru Buium, Netan Dogra, Paolo Mantero, and Arnab Saha. 
We also warmly thank Alexandru Buium for pointing out a mistake in a preliminary version and Netan Dogra for alerting us to \cite{Scarponi:Sparsity}. 
The second author is grateful for the support of the U.S.~National Science Foundation (DMS-2418796) and the Simons Foundation (MPS-TSM-00007985).

\section{Preliminaries} 
\label{sec:prelims}
We give a self-contained section summarizing the needed background as well as establishing notation and conventions. 

\subsection{Conventions}
\label{subsec:Conventions}
We set the following conventions. 

\subsubsection{\bf Fields and rings}
\label{subsec:field_conventions}
Throughout, $p$ will always be a prime number which unless otherwise stated is assumed to be {\it odd}. 
We set $k := \mF_p^{\alg}$,  $R := W(\mF_p^{\alg})$ the completed ring of integers of the maximal unramified extension of $\mQ_p$, and $K := R[1/p]$ the fraction field of $R$. 

\subsubsection{\bf Algebraic geometry}
For any ring $B$, a \cdef{$B$-variety} $X$ will always be an integral, separated scheme of finite type over $\Spec(B)$ where the structure morphism $\pi\colon X \to \Spec(B)$ is flat; this is automatic when $B$ is a field. 
A \cdef{$B$-subvariety} $Y$ of a $B$-variety $X$ will always be a closed $B$-subvariety of $X$. For a ring $S$, we set $X(S) := \Hom(\Spec S,X)$. 
Additionally, for a smooth $L$-variety $X$, we let $\Omega_X^1$ (resp.~$T_X$) denote the cotangent (resp.~tangent) bundle of $X$, which is a vector bundle of rank $\dim(X)$. 
By definition, we have that $\Omega_X^{1,\vee}\cong T_X$.

When $B = k$, as in Subsection \ref{subsec:field_conventions}, and $X$ a $k$-variety, we let $F_X \colon X \to X$ denote the absolute Frobenius morphism, and for each $e \geq 0$, we set $F_{X}^e \colon X \to X$ to be the $e$-th iterate of $F_X$. 
We will also describe this as the $e$-fold Frobenius morhpism. 
For $B = R$ as in Subsection \ref{subsec:field_conventions} and a $R$-variety $X$, we let $X_0$ denote the special fiber of $X$ (i.e., the base change of $X$ along $\Spec(k) \to \Spec(R)$) and $X_{\eta}$ denote the generic fiber of $X$ (i.e., the base change of $X$ along $\Spec(K) \to \Spec(R)$).

At times we will utilize formal schemes, specifically over the base ring $R$ as above. As such, we tacitly assume all formal schemes are topologically of finite type and $p$-torsion free. 
When possible we reserve calligraphic letters, e.g., $\cX, \cY,$ etc.~for formal schemes. 
We will use fraktur letters, e.g., $\mathfrak{X}$, $\mathfrak{Y}$, etc.~when spreading out varieties defined over a number field to a localization of the ring of integers. 
For $G$ a group scheme or group formal scheme, we denote by $G[\tors]$ its torsion subgroup, and fixing a prime $p$ by $G[p^\infty]$
the $p$-power torsion, and $G[\nptors]$ the prime-to-$p$ torsion. 

\subsection{Vector bundles} 
Much of our analysis concerns various basic facts about positivity of vector bundles, see \cite{Hartshorne:AmpleVectorBundle} for a general review of the notion of ampleness for vector bundles. 
Throughout this subsection, for $\sF$ a vector bundle, we denote by $\mP(\sF)$ its projectivization. We recall the needed ingredients. 

\begin{lemma}\label{lemma:ample_vb_curve}
For $\sF$ an ample vector bundle on a smooth projective $k$-curve, 
\[
\deg(\sF) = \deg(\det(\sF)) > 0.
\]
\end{lemma}

\begin{proof}
This follows from three statements:~$\det(\sF)$ is an ample line bundle by  \cite[Corollary 2.6]{Hartshorne:AmpleVectorBundle}, $\deg(\sF) = \deg(\det\sF)$ by \cite[\href{https://stacks.math.columbia.edu/tag/0DJ5}{Tag 0DJ5}]{stacks-project}, and finally an ample line bundle on a curve has strictly positive degree \cite[\href{https://stacks.math.columbia.edu/tag/0B5X}{Tag 0B5X}]{stacks-project}. 
\end{proof}

\begin{lemma}\label{lemma:amplecotangent_subvarieties}
Let $L$ be a field,  and let $X$ be a smooth projective $L$-variety with ample cotangent bundle. 
\begin{enumerate}
\item If $\cchar L = 0$, then every integral subvariety of $X$ is of general type. 
\item If $\cchar L = p > 0$, then every integral subvariety of $X$ with dimension 1 is of general type. 
\end{enumerate}
\end{lemma}

Before the proof, we recall that an integral projective variety is of  \cdef{general type} if it admits a desingularization which is of general type. 
Such a desingularization exists in arbitrary dimension when $\cchar L = 0$ by Hironaka's celebrated resolution of singularities and for dimension 1 when $\cchar L = p>0$ by classic techniques. 

\begin{proof}
The characteristic $0$ statement is classical, see e.g., \cite[6.3.28]{Lazarsfeld:Positivity2}. 
The characteristic $p$ statement for rational curves can be found in \cite[Proposition 5.(2)]{MartinDeschamps:DescentCotangentAmple}. 
We note that the same argument from \textit{loc.~cit.}~works for genus 1 curves. 
\end{proof}

\subsection{Geometry of projective bundles}\label{sec:bundles} 
Many of the $k$-varieties $X$ we need are principal homogeneous spaces under $F^*_X \Omega_{X}^1$. As such, they correspond naturally to classes in $H^1(X,F^*_X \Omega_{X}^1)$ which we interpret as vector bundle $\sE$ fitting into the extension $0\to \sO_X \to \sE \to F^*_X \Omega_{X}^1 \to 0$. 
We recall key properties about such extensions and their projectivizations due to \cite{MartinDeschamps:DescentCotangentAmple}.

\begin{prop}[\protect{\cite[Proposition 1 \& 2]{MartinDeschamps:DescentCotangentAmple}}]
\label{prop:splitting_sequence_pullback}
Let $Y$ be a $L$-variety and $\sF$ a vector bundle on $Y$. 
Suppose we have a short exact sequence of vector bundles 
\begin{equation}\label{eqn:SES_standard}
0\to \sO_Y \to \sE \to \sF \to 0
\end{equation}
over $Y$.  We have the following:
\begin{enumerate}
\item The closed subscheme $\mathbf{P}(\sF)$ is a divisor of $\mathbf{P}(\sE)$ where $\sO(\mathbf{P}(\sF))\cong \sO_{\mathbf{P}(\sE)}(1)$. 
\item The complement $\mathbf{P}(\sE) \setminus \mathbf{P}(\sF)$ is a $\RSpec \Sym \sF$-torsor that represents the functor $R_{\sE}$ corresponding to splittings of the sequence \eqref{eqn:SES_standard}. 
\item If $Y$ is proper and $\sF$ is ample,  then for any closed subscheme $Z$ of $\mathbf{P}(\sE)$ that is disjoint for $\mathbf{P}(\sF)$, there exists a finite number of points $\{z_1,\dots,z_n\}$ of $Z$ such that the projection morphism $\pi\colon Z \to Y$ is finite and radical outside of $\{ z_1,\dots,z_n\}$. In other words, $\pi_{|Z}\colon Z \to X$ is finite and generically radical. 
\item If $\sE$ is not ample, then there exists a positive dimension closed subscheme $Z$ of $\mathbf{P}(\sE)$ as in (3) with the property that the short exact sequence \eqref{eqn:SES_standard} splits after pullback to $Z$.  
Moreover, if $\sE$ is ample, then there does not exist such a $Z$. 
\end{enumerate}
\end{prop}

\begin{cor}[\protect{\cite[Corollaire 3]{MartinDeschamps:DescentCotangentAmple}}]
\label{coro:splitting_after_frob_curves}
Let $C$ be a smooth proper $k$-curve, and suppose we have a short exact sequence of vector bundles 
\begin{equation}\label{eqn:SES_standard2}
0\to \sO_C \to \sE \to \sF \to 0
\end{equation}
over $C$ where $\sF$ is ample. 
Either $\sE$ is ample or the there exists some $n\ge 1$ such that pullback of the sequence \eqref{eqn:SES_standard2} under the $n$-fold Frobenius morphism splits. 
\end{cor}

\subsection{Intersection theory}\label{sec:intersection_theory} 
We will need to make various standard calculations in intersection theory to generalize the effective part of \cite{Buium:GeometrypJets}. This material is standard, but we give an overview with references for the less initiated. The primary reference is the standard \cite{Fulton:IntersectionTheory} and \cite[Tag 0AZ6]{stacks-project}, but we note differences in the exposition here. 
Recall the conventions from Subsection \ref{subsec:field_conventions}.

\subsubsection{\bf Chow groups and rings}
For $X$ a smooth projective $k$-variety, we denote by $A_j(X)$ additive group of formal $j$-cycles on $X$ up to rational equivalence. 
A $j$-cycle is rationally equivalent to zero if it has the form $\sum [\div(f_i)]$ for rational functions $f_i$ on $(j+1)$-dimensional subvarieties $W_i \subset X$. 
As such, the elements of $A_j(X)$ represented by finite formal sums $\sum n_i [V_i]$ for $V_i \subset X$ subvarieties of dimension $j$ and $n_i \in \ZZ$ for all $i$ and the group operation is formal addition. 
The sum $A_*(X) := \bigoplus A_j(X)$ is the Chow group. 

The Chow group also comes equipped with an intersection product 
\begin{align*}
A_r(X) \times A_s(X) &\to A_{r+s-\dim(X)}(X)\\
(\alpha,\beta) &\mapsto \alpha\cdot \beta
\end{align*}
where $\alpha\cdot \beta$ is defined as in \cite[\href{https://stacks.math.columbia.edu/tag/0B0G}{Tag 0B0G}]{stacks-project}. 
This product on $A_*(X)$ is commutative, associative, and has a unit, and hence we may refer to $A_*(X)$ as the Chow ring of $X$. 

\subsubsection{\bf 0-cycles}
\label{subsec:0cycles}
When referencing $0$-cycles on $X$ i.e., elements of $A_0(X)$, we will simply identify them with their image under the degree map \cite[\href{https://stacks.math.columbia.edu/tag/0AZ0}{Tag 0AZ0}]{stacks-project}. 
Moreover, $0$-cycles will be identified with integers.

\subsubsection{\bf Chern classes}
Fixing a Weil divisor $D$ on $X$ and $\iota \colon V \hookrightarrow X$ an embedded subvariety of dimension $j$, the intersection class is given by $D \cdot [V] := [\iota^* D] \in A_{j-1}(X)$ which can extend linearly to any $\alpha \in A_j(X)$. This is commonly written as a homomorphism $c_1(\sL) \cap \Arg$ given by $c_1(\sL) \cap \alpha:= D\cdot \alpha$. The symbol $c_1(\sL)$ is the \cdef{first Chern class} of $\sL$. It satisfies basic commutativity and additivity properties allowing for the evaluation of $P(c_1(\sL_1),\ldots,c_1(\sL_n)) \cap \alpha$ for any homogeneous polynomial $P$, line bundles $\sL_1,\ldots,\sL_n$, and $j$-cycle $\alpha \in A_j(X)$. This leads to the degree $\deg_{\sH}(\alpha)$ for any very ample line bundle $\sH$. Indeed, use $\sH$ to define an embedding to $\PP^n$ and compute the degree of the zero cycle $c_1(\sH)^j \cap \alpha$. 

\subsubsection{\bf Segre classes}
\label{subsub:Segreclasses}
For a vector bundle $\sE$ over $X$ of rank $e+1$, we obtain its projectivization $p \colon \PP(\sE) \to X$ with tautological line bundle $\sO(1)$. From this, we have the $i$-\cdef{th Segre class} $(-1)^is_i(\sE)$ considered as the homomorphism\footnote{In \cite[Section 3.1]{Fulton:IntersectionTheory}, Fulton defines Segre classes using $\mP(\sE^{\vee})$. The sign in our convention accounts for this discrepancy as $s_i(\sE^{\vee}) = (-1)^is_i(\sE)$ (cf.~\textit{loc.~cit.}~Remark 3.2.3.(a))} $(-1)^is_i(\sE) \colon A_j(X) \to A_{j-i}(X)$ given by $$\alpha \mapsto (-1)^ip_*( c_1(\sO(1))^{e+i} \cap p^* \alpha).$$ 

For $X$ of dimension $n$, we see that the $n$-th Segre class is non-zero precisely when considered as the homomorphism $(-1)^ns_n(\sE)\colon A_n(X) \to A_0(X)$, and since $A_n(X) = \mZ\cdot [X]$, we will simply identify $(-1)^ns_n(\sE)$ with the image $(-1)^ns_n(\sE)([X]) \in A_0(X)$, which we consider to be an integer via Subsection \ref{subsec:0cycles}. 

The \cdef{total Segre class of $\sE$} is defined to be the formal sum
\[
s(\sE) = \sum_{i\geq 0}(-1)^is_i(\sE).
\]
A priori, this is an infinite sum, but it is actually finite as $A_j(X)$ is non-zero for finitely many $j$. 

To conclude our subsection on intersection theory, we record two useful facts, to be used later, concerning Segre classes.

\begin{lemma}
\label{lemma:SegreChernInverse}
For a vector bundle $\sE$ of rank $e+1$ on $X$, the total Chern class of $\sE$ is the inverse of the total Segre class of $\sE$. 
More precisely, we have the following equalities in $\CH(X)$:
\begin{align*}
c_1(\sE) &= (-1)s_1(\sE) \\
c_2(\sE) &= ((-1)s_1(\sE))^2 - s_2(\sE).  
\end{align*}
We will adopt the second equality as the definition of the second Chern class of $\sE$. 
\end{lemma}

\begin{proof}
This is \cite[Section 3.2]{Fulton:IntersectionTheory}. 
\end{proof}

\begin{lemma}
\label{lemma:SegreWhitneySum}
For
\[
0 \to \sE' \to \sE \to \sE'' \to 0
\]
a short exact sequence of vector bundles on a projective $k$-variety $X$, $s(\sE) = s(\sE')\cdot s(\sE'')$.
\end{lemma}

\begin{proof}
This follows from the well-known Whitney summation formula for total Chern classes \cite[Theorem 3.2.(e)]{Fulton:IntersectionTheory} and the fact that the total Chern class is the inverse of the total Segre class.
\end{proof}

\subsection{Arithmetic jets} 
As with Buium's original proof, the proof of our main theorems will utilize arithmetic jet spaces in a critical way. We provide a review, leaving a more detailed treatment to \cite{Bui05}. 
We remind the reader of our conventions from Subsection \ref{subsec:field_conventions}, namely that $p$ denotes a fixed odd prime integer.

\subsubsection{\bf $p$-derivatives}
For $u \colon A \to B$ a homomorphism of rings, a \cdef{$p$-derivation with respect to $u$} is a set map $\delta \colon A \to B$ so that 
\begin{align*}
\delta(x+y) &  = \delta(x) + \delta(y) + (1/p)( u(x)^p + u(y)^p - (u(x)+u(y))^p), \textrm{ and } \\
\delta(xy) & = u(x)^p \delta(y) + u(y)^p \delta(x) + p \delta(x) \delta(y).
\end{align*} 
When $A = B$ and $u$ is identity, we simply say that $\delta \colon A \to A$ is a \cdef{$p$-derivation}. A ring $A$ with a choice of $p$-derivation $\delta$ is called a \cdef{$\delta$-ring}. These identities in the definition are precisely those needed so that the map $\phi \colon A \to A$ given by $\phi(x) = x^p + p \delta(x)$ is a ring homomorphism. Such homomorphisms are called \cdef{lifts of Frobenius}. When $A$ is $p$-torsion free, every lift of Frobenius defines a $p$-derivation. One should consider $p$-derivations as arithmetic analogues of derivations. Indeed, \cite{Bui00} shows that under reasonable assumptions, any derivation like function must be either a derivation or a $p$-derivation for a particular prime. 

\subsubsection{\bf Affine prolongation sequence}
Fix a $\delta$-ring $A$. An \cdef{affine prolongation sequence} is sequence $S^*$ of noetherian $A$-algebras $\{ S^r \}$ for $r \geq 0$ together with $A$-algebra homomorphisms $u^r \colon S^r \to S^{r+1}$ and choices of $p$-derivations $\delta^r \colon S^r \to S^{r+1}$ with respect to $u^r$ so that $u^{r+1} \delta^r = \delta^{r+1} u^r$. Note, the superscripts are not powers, but convey the source and target of the maps. It is common to suppress the superscripts and simply refer to these as $u$ and $\delta$ letting the source and target determine context. A basic example is the constant sequence $S^*$ with $S^r := A$ with the choices $u = \id$ and $\delta$ the fixed $p$-derivation on $A$. Affine prolongation sequences naturally form a category $\AffProl$ with natural maps $\eta^* \colon S^* \to T^*$ consisting of level-wise homomorphisms $\eta^r \colon S^r \to T^r$ for each $r$ so that the relevant squares commute formed with $\eta$ and $u$ or $\delta$ commute. 

\subsubsection{\bf Arithmetic jet algebras of order $r$}
Now fix $A = R$ as in Subsection \ref{subsec:field_conventions} with its unique structure as a $\delta$-ring. Let $S = R[\ux]/I$ be an affine $R$-variety of finite type, so $\ux = x_1,\ldots,x_n$ is a tuple of variables. Fix $r \geq 0$ and set $\ux',\ux'',\ldots,\ux^{(r)}$ to be new sets of variables. The notation is invocative of differentiation and indeed one builds from this an affine prolongation by declaring $S^r := R[\ux',\ux'',\ldots,\ux^{(r)}]$ and using the structure maps $S^r \to S^{r+1}$ as the natural inclusions and $p$-derivations $S^r \to S^{r+1}$ given by sending $\delta(\ux^{(i)}) := \ux^{(i+1)}$. This can be extended to include the ideal $I$. 
The \cdef{$r$-th arithmetic jet algebra} is the $R$-algebra $$J^r S :=  R[\ux,\ux',\ldots,\ux^{(r)}]/(I,\delta I, \delta^2 I,\ldots, \delta^r I).$$ Here, by $\delta^i I$ we mean to pick generators $I = (f_1,\ldots,f_s)$ and compute $$\delta^i I := (\delta^i f_1,\ldots, \delta^r f_s) \subset R[\ux,\ux',\ldots,\ux^{(r)}]$$ by the set theoretic identities defining $\delta$. This is easily checked not to depend on the choice of generators, and the package $\{J^r S\}$ forms an affine prolongation sequence. 

\subsubsection{\bf Arithmetic jet spaces of order $r$}
Considering $X = \Spec(S)$, one can refer to $\Spec (J^r S)$ as the $r$-th arithmetic jet scheme. Globalizing this construction is subtle owing to the simple fact that this construction does not commute with localization. Indeed, for $f \in S$ and $g/f^n \in S_f$ in a principal localization, $\delta(g/f^n)$ is not a rational function but a $p$-adically convergent power series in $\delta(f)/f^p$. Thus, the most common approach to globalization invokes $p$-adic completion, but this is not the only way \cite{Bor11,BPS23}. Thus we consider the \cdef{$r$-th arithmetic jet space} as the {\it formal scheme} $\Spec (J^r S)^{\wedge}$, where $(\Arg)^{\wedge}$ will always denote $p$-adic completion. Now, given any $R$-scheme $X$ locally of finite type, we may construct $J^r X$ for any $r\geq 1$ by gluing along affine covers.

We may package these arithmetic jet spaces as follows. 
One can introduce a category $\Prol_R$ whose objects are sequences of $R$-formal schemes $\cY^*$ together with morphisms $\cY^r \to \cY^{r-1}$ for each $r \geq 0$, for which the system $\sO(\cY^r)$ forms an affine prolongation sequence. The constant sequence $\cY^r := \Spf (R)$ forms naturally an object, which we denote as $R^* \in \Prol_R$. As the notation suggests, each object $\cY^*$ admits a natural map $\cY^* \to R^*$ and morphisms commute with these structure maps. 

The arithmetic jet spaces $J^* X$ for a $R$-scheme $X$ locally of finite type also define an object in $\Prol_R$, which satisfies a particularly important universal property. 

\begin{theorem}\label{thm:univpropjets} For $X$ a $p$-complete $R$-scheme locally of finite type and $\cY^*$ a prolongation sequence, there is a natural bijection $\Hom_{\FrSch_R}(\cY^0,X^{\wedge}) \cong \Hom_{\Prol_R}(\cY^*, J^* X)$.
\end{theorem}

\begin{proof}
This is easily checked by reducing the affine case, which is \cite[Proposition~3.3]{Bui05}.
\end{proof}

\subsubsection{\bf Special fiber of first order arithmetic jet space}
\label{subsub:specialfiber}
For $X$ a $R$-variety with its special fiber $X_0$, the \cdef{Frobenius tangent space} is the space $\FT(X_0/k) := \RSpec \Sym F_{X_0}^* \Omega_{X_0}^1$.
By \cite[Proposition 1.5]{Buium:GeometrypJets} and \cite[Proposition 3.9]{Bui05}, we have that when $X$ is smooth along $X_0$, the projection $(J^1 X)_0 \to X_0$ is an $\FT(X_0/k)$-torsor. 
Thus, many results and calculations in Subsection~\ref{sec:bundles} apply to this situation. 
We will leverage such results in Section~\ref{sec:affinenes}. 

The description of $(J^1 X)_0 \to X_0$ as a $\FT(X_0/k)$-torsor will be heavily used in our proofs, and so we wish to expand on this. 
The torsor description of $(J^1 X)_0$ allows us to consider this space as an element of $H^1(X_0,F_{X_0}^*T_{X_0})$ where $T_{X_0}\cong \Omega_{X_0}^{1,\vee}$,  which we realize as a vector bundle $\sE_{X_0}$ sitting in the short exact sequence
\[
0\to \sO_{X_0}\to \sE_{X_0} \to F_{X_0}^*\Omega_{X_0}^1 \to 0.
\]
Using \autoref{prop:splitting_sequence_pullback}.(1,2), we may consider the divisor $D_{X_0} := \mP(F^*_{X_0} \Omega_{X_0}^1)$ in $\mP(\sE_{X_0})$, which belongs to the linear system of $\sO_{\mP(\sE_{X_0})}(1)$. 
Moreover, the $\FT(X_0/k)$-torsor $(J^1X)_0$ is identified with $\mP(\sE_{X_0})\setminus D_{X_0}$.

\begin{remark}
We remark that our construction of $(J^1 X)_0$ differs slightly from \cite{Buium:GeometrypJets}. 
Buium considers the base change of $X$ to $\Spec(R/p^2R)$ and defines the first order jet space of this base change and shows that this admits a map to $X_0$. This consideration is largely present to simplify the globalization of the arithmetic jet construction. The above references show that the constructions here produce the same object. In order to keep our exposition consistent with Buium's, we set $J^n(X_0) \coloneqq (J^nX)_0$. 
\end{remark}

\subsubsection{\bf A universal property and the nabla map}
\label{subsub:nabla}
The primary tool to extract from the universal property described in \autoref{thm:univpropjets} is a point lifting principal. Fix $X$ an $R$-scheme locally of finite type,  and let $P$ be an $R$-point of $X$ viewed as a map $\Spec(R) \to X$. This naturally determines a map $\Spf(R) \to X^{\wedge}$, and \autoref{thm:univpropjets} determines a map of prolongations $R^* \to J^*X$. Fixing $r \geq 1$, this defines a map $\Spf(R) \to J^r X$ and thus we have an injective set map $$\nabla^r \colon X(R) \to J^rX(R).$$ We will often drop $r$ from the notation, preferring $\nabla$. We additionally consider reduction to the special fiber which induces the map 
\[
\nabla_0 \colon X(R) \to J^r(X_0)(k).
\]
It is the map $\nabla_0$ which forms the critical tool driving the types of Diophantine applications extracted from arithmetic jets.

\subsubsection{\bf Arithmetic jet spaces of abelian varieties}
We record one generalization of a claim used in \cite{Buium:GeometrypJets} to higher dimensions. 

\begin{lem}\label{lemma:maximal_abelian_subvariety} 
For an abelian $R$-scheme $A$ and any $n \geq 1$, the variety $B := p J^n(A_0) \subset J^n(A_0)$ is the maximal abelian subvariety and $\dim B = \dim A_0$.
\end{lem}

\begin{proof} 
For the first claim, it suffices to show $B$ is proper. Along the way, we show that $B$ is isogenous to $A_0$ and thus the second claim will also follow. 

As above, we may consider the special fiber of the $n$-th arithmetic jet space for $A$, which comes with a natural projection map $\pi_n\colon J^n(A_0) \to A_0$. 
From \cite[Propositions 1.5 \& 2.4]{Buium:GeometrypJets}, we have that the kernel of the projection $\pi_n$ is a power of the additive group scheme $\bG_a^{n \dim A_0}$ over $k$. 
Consider the multiplication by $p$ map $[p] \colon J^n(A_0) \to p J^n(A_0)$. As multiplication by $p$ on $\bG_a^{n \dim A_0}$ is $0$, we have that $\ker \pi_n \subset \ker [p]$, which forces the surjective projection to descend to a surjective map $A_0 \to p J^n (A_0)$, as indicated in the following diagram.
\[
\begin{tikzcd}
0 \arrow{r} & \bG_a^{n \dim A_0} \arrow[r, hook]{r} & J^n(A_0) \arrow[two heads]{r} \arrow{d}{[p]} & A_0 \arrow[two heads]{ld} \arrow{r} & 0 \\
{} & {} & pJ^n(A_0) & {} & {} 
\end{tikzcd}
\]

As $pJ^n(A_0)$ is separated and locally of finite type, properness descends \cite[\href{https://stacks.math.columbia.edu/tag/03GN}{Tag 03GN}]{stacks-project}, and hence $pJ^n(A_0)$ is an abelian $k$-variety.  For the dimension claim, suppose $\ker(A_0 \to pJ^n(A_0))$ is not a finite group. It is then an abelian $k$-subvariety of $A_0$ which maps to zero in $pJ^n(A_0)$. Additionally, its inverse image in $J^n(A_0)$ is a subvariety of $J^n(A_0)$ which maps to zero under multiplication by $p$, and so torsion. 
Thus, we have a contradiction and hence $\ker(A_0 \to pJ^n(A_0))$ is finite and so $\dim pJ^n(A_0) = \dim A_0$. 
\end{proof}

\section{Affineness of arithmetic jets} 
\label{sec:affinenes}
In this section, we prove our \autoref{thm:main_Affine} concerning the affineness of the special fiber of arithmetic jet spaces for $R$-varieties whose special fiber has ample cotangent bundle.  

\subsection{Preparations}
We remind the reader of our conventions from Subsection \ref{subsec:Conventions}. 
Our proof requires a critical result from Tango and two lemmas.

\begin{theorem}[\protect{\cite[Theorem 25]{Tango:ExtensionsVBFrobenius}}]
 \label{thm:Tango_condition}
 Let $\sF$ be any indecomposable vector bundle of rank $r$ on a smooth projective $k$-curve $C$ of genus $g >1$. 
 If 
 \[
 \deg(\sF) >r(r-1) + (g-1) + r\left( \frac{2g-2}{p}\right),
 \]
 then the induced map $H^1(C,\sF^{\vee}) \hookrightarrow H^1(C,F_C^*\sF^{\vee})$ is injective.  
 \end{theorem}

The next lemma illustrates how the structure of the first arithmetic space behaves under restriction to closed subschemes.

\begin{lemma}\label{lemma:pullback_split_J1_split}
Let $X$ be a smooth projective $R$-variety with special fiber $X_0$ having ample cotangent bundle $\Omega_{X_0}^1$, let $C$ be a smooth projective $R$-curve, and let $\iota\colon C \hookrightarrow X$ be a closed immersion which induces a closed immersion on the special fiber $\iota_0\colon C_0 \hookrightarrow X_0$. 
Let $\eta_{X_0} \in H^1(X_0,F_{X_0}^*T_{X_0})$ denote the class corresponding to the $\FT(X_0/k)$-torsor $J^1(X_0) \to X_0$, hence $\eta_{X_0}$ corresponds to a short exact sequence of vector bundles
\begin{equation}\label{eqn:SES_lemma}
0\to \sO_{X_0} \to \sE \to F^*_{X_0}\Omega_{X_0}^1 \to 0
\end{equation}
where $\sE$ is a vector bundle over $X_0$.

If the pullback of \autoref{eqn:SES_lemma} along $\iota_0$ splits,
the extension 
\[
0\to \sO_{C_0} \to \sF \to F^*_{C_0}\Omega_{C_0}^1 \to 0
\]
corresponding to the class $\eta_{C_0} \in H^1(C_0,F^*_{C_0}T_{C_0})$ realizing the $\FT(C_0/k)$-torsor $J^1(C_0) \to C_0$ is also split. 
\end{lemma}

\begin{proof}
Consider the dual of \autoref{eqn:SES_lemma} and pull back this sequence along $\iota_0$ to get a split short exact sequence
\begin{equation}\label{eqn:SES_lemma_dual}
0\to \iota_0^*F^*_{X_0}T_{X_0} \to \iota_0^*\sE^{\vee} \to \sO_{C_0}\to 0.
\end{equation}
We claim that the map $d\iota_{0*}\colon H^1(C_0,T_{C_0}) \to H^1(C_0,\iota_{0}^*T_{X_0})$ is injective. Since $C_0$ and $X_0$ are smooth, we have a short exact sequence
\[
0 \to T_{C_0} \to \iota_0^*T_{X_0} \to \sN_{C_0/X_0} \to 0
\]
where $\sN_{C_0/X_0}$ is the normal bundle of $\iota_0$. 
Taking cohomology of this short exact sequence, we have
\[
0 \to H^0(T_{C_0})\to H^0(\iota_0^*T_{X_0}) \to H^0(\sN_{C_0/X_0}) \to H^1(T_{C_0}) \to H^1(\iota_0^*T_{X_0}) \to H^1(\sN_{C_0/X_0}) \cdots
\]
where we have dropped the $C_0$ from the above cohomology groups.

As $C_0$ has genus at least $2$, we see that $ H^0(T_{C_0}) = 0$ and the amplitude of $\Omega_{X_0}^1$ implies that $H^0(\iota^*T_{X_0}) = 0$.  
Indeed, the latter claim follows from the same argument as in \cite[Remark 6.1.4 \& Corollary 6.3.30]{Lazarsfeld:Positivity2}. 
We claim that $H^0(\sN_{C_0/X_0}) = 0$ i.e., $C_0$ is rigid in $X_0$. 
As $C_0$ and $X_0$ are smooth, \cite[\href{https://stacks.math.columbia.edu/tag/0E9K}{Tag 0E9K}]{stacks-project} implies that $\iota_0$ is a local complete intersection morphism and hence $\sN_{C_0/X_0}$ is a vector bundle. 
Via \cite[Proposition 2.4]{Debarrebook}, we know that $\iota_0^*T_{X_0} \cong \sN_{\Gamma/C_0\times X_0}$ where $\sN_{\Gamma/C_0\times X_0}$ is the normal bundle of the graph $\Gamma$ of $\iota_0$ inside $C_0\times X_0$. 
Let $\pi_2\colon C_0\times X_0 \to X_0$ denote the second projection. 
Using smoothness, we have that $\pi_2^*\sN_{C_0/X_0}\cong \sN_{\Gamma/C_0\times X_0}$. 
Since $\sN_{C_0/X_0}$ is a vector bundle and $\pi_2$ is a surjective, proper $k$-morphism, we have that $H^0(\sN_{C_0/X_0}) $ injects into $ H^0(\pi_2^*\sN_{C_0/X_0}) = H^0( \sN_{\Gamma/C_0\times X_0}) = H^0(\iota_0^*T_{X_{0}}) = 0$.  
Therefore, the above sequence reduces to 
\[
0 \to H^1(C_0,T_{C_0}) \to H^1(C_0,\iota_0^*T_{X_0}) \to H^1(C_0,\sN_{C_0/X_0}) \cdots. 
\]
This then tells us that the kernel of $d\iota_{0*}\colon H^1(C_0,T_{C_0}) \to H^1(C_0,\iota_{0}^*T_{X_0})$ is zero hence injective.

To finish out claim, we recall that \cite[Theorem 1.2]{DupuyDZB:DeligneIlluse} implies that $\iota_0^*\eta_{X_0}\cong d\iota_{0*}\eta_{C_0}$ . 
Now,  as \autoref{eqn:SES_lemma_dual} is split i.e., $\iota_0^*\eta_{X_0}$ corresponds to the identity element in $H^1(X_0,\iota_0^*F^*_{X_0}T_{X_0})$,  injectivity of $d\iota_{0*}$ tells us that the class $\eta_{C_0} \in H^1(C_0,F^*_{C_0}T_{C_0})$ must be trivial, and hence correspond to a split extension. 
\end{proof}

Finally, we show how ampleness of the cotangent bundle of the special fiber influences the Frobenius tangent space torsor structure of the special fiber of the first arithmetic jet space. 

\begin{lemma}\label{lemma:ample_implies_nonsplit}
Let $X$ be a smooth projective $R$-variety such that $X_0$ has ample cotangent bundle. 
Then, the $\FT(X_0/k)$-torsor $J^1(X_0) \to X_0$ is non-trivial, equivalently the extension
\[
0\to \sO_{X_0} \to \sE \to F^*_{X_0}\Omega_{X_0}^1 \to 0.
\]
does not split.
\end{lemma}

\begin{proof}
By \cite[Corollary 2.6]{Hartshorne:AmpleVectorBundle}, the ampleness of $\Omega_{X_0}^1$ implies that $K_{X_0} = \det(\Omega_{X_0}^1)$ is also ample, and hence $X_0$ is of general type. 
By \cite[Proposition 3.2.1.(c)]{Achingeretal:GlobalFrobLiftability1}, we have that $X_0$ is not infinitesimally trivial, and hence \cite[Proposition 1.7]{Buium:GeometrypJets} implies our claim. 
\end{proof}

\subsection{Proof of affineness result}
With the above results, we are now in a position to prove \autoref{thm:main_Affine}, which we recall below for the readers convenience. 
In the same spirit as \cite{Buium:GeometrypJets}, we aim to show the vector bundle corresponding to the extension class of $J^1(X_0) \to X_0$ is ample. 
We accomplish this by studying the behavior of this vector bundle when restricted to curves inside of our variety. 

\begin{theorem}[= \autoref{thm:main_Affine}]\label{thm:main0}
Let $X$ be a smooth projective $R$-variety such that $X_0$ has ample cotangent bundle. Suppose there exists a smooth projective $R$-curve $C$ of genus $g \geq 2$ with a closed embedding $\iota\colon C \hookrightarrow X$ and that $p>2\dim(X)^2g$. For every $n\geq 1$, the special fiber $J^n(X_0)$ of the arithmetic $n$-jet space $J^nX$ is affine. 
\end{theorem}

\begin{proof} 
Throughout, we only consider the special fiber, so we set $X = X_0$, $C = C_0$, and $\iota = \iota_0$ to limit subscripts. 
We know that $J^1X \to X$ is a $\FT(X/k)$-torsor and hence corresponds to an extension
\begin{equation}\label{eqn:SES_J1}
0\to \sO_{X} \to \sE \to F^*_X\Omega_{X}^1 \to 0.
\end{equation}
We record two facts about \autoref{eqn:SES_J1}. 
By \autoref{lemma:ample_implies_nonsplit} and our assumption that $X_0$ has ample cotangent bundle, we know that \autoref{eqn:SES_J1} is non-split. 
Also, note that \cite[Proposition 4.3]{Hartshorne:AmpleVectorBundle} asserts that $F^*_X\Omega_X^1$ is ample.

Our goal is to show that $\sE$ is ample.
To do so, we will suppose to the contrary that $\sE$ is not ample. 
Our analysis is based on the restriction of $\sE$ to curves in $X$. 
Let $\iota_C\colon C \hookrightarrow X$ denote any closed curve in $X$. 
Consider the pullback of \autoref{eqn:SES_J1} along $\iota_C$:
\begin{equation}\label{eqn:pullback_SES_J1}
0\to \sO_{C} \to \iota_C^*\sE \to \iota_C^*F^*_X\Omega_{X}^1 \to 0.
\end{equation}
Again, we have that \cite[Proposition 4.3]{Hartshorne:AmpleVectorBundle} implies that $\iota_C^*F^*_X\Omega_X^1$ is ample. By \autoref{prop:splitting_sequence_pullback}.(2,3), there exists a positive dimensional closed subscheme $Z$ of $\mathbf{P}(\sE)$ such that $\pi_{|Z}\colon Z \to X$ is finite and generically radical, and the short exact sequence \autoref{eqn:SES_J1} splits after pullback to $Z$ along $\pi_{|Z}$. 
Consider the following Cartesian diagram
\[
\begin{tikzcd}
C' \arrow{r}{\pi_{|C'}} \arrow[hook]{d}{\iota_{C'}} & C \arrow[hook]{d}{\iota_C}\\
Z \arrow{r}{\pi_{|Z}} & X
\end{tikzcd}
\]
where $C'$ is an irreducible component of the pullback $C \times_{\iota_C,X,\pi_{|Z}} Z$ endowed with its reduced induced scheme structure. 
Note that $C'$ is then an integral $k$-curve and $\pi_{|C'}\colon C' \to C$ is a non-constant, hence surjective $k$-morphism. 
Consider the normalizations $\nu'\colon \tilde{C}' \to C'$ of $C'$ and $\nu\colon \tilde{C} \to C$ of $C$.  
The $k$-morphism $\pi_{|C'}$ will induce a finite and generically radical map
\[
\tilde{\pi_{|C'}}\colon \tilde{C}' \to \tilde{C}. 
\]
Since $\tilde{\pi_{|C'}}$ induces a purely inspearable map on function fields, \cite[\href{https://stacks.math.columbia.edu/tag/0CCZ}{Tag 0CCZ}]{stacks-project} implies $\tilde{\pi_{|C'}}$ is the $n$-fold Frobenius map for some $n\geq 1$.

We will consider two cases, each case resulting in a contradiction which means that $\sE$ must be ample. 

\

\noindent {\it Case 1:}~Suppose there is a closed curve $C$ in $X$ such that $\sE_{|C}$ is ample. 
Note that since $\nu$ is finite, \cite[Proposition~4.3]{Hartshorne:AmpleVectorBundle} says that $\nu^*\sE_{|C}$ is ample. 
By \autoref{prop:splitting_sequence_pullback}.(4) and the commutativity of the above diagram implies that the pullback of \autoref{eqn:SES_J1} along $(\pi_{|Z}\circ \iota_{C'}\circ \nu')$ will split and therefore, the pullback of \autoref{eqn:SES_J1} along $(\tilde{\pi_{|C'}}\circ \nu \circ \iota_C)$ will also split. Thus, the pullback of sequence \autoref{eqn:pullback_SES_J1} along $\nu$ splits after pullback along a suitable iterate of Frobenius. This contradicts \autoref{coro:splitting_after_frob_curves} as $\nu^*\sE_{|C}$ is ample since $\sE_{|C}$ was assumed to be ample.

\

\noindent {\it Case 2:}~Now suppose that for \textit{every} curve $C$ in $X$, $\sE_{|C}$ is not ample.  
Let $\iota_C \colon C \hookrightarrow X$ denote the closed immersion from the assumption.  
As above, consider the pullback of \eqref{eqn:SES_J1} along $\iota_C$:
\begin{equation}\label{eqn:pullback_SES_J1_2}
0\to \sO_{C} \to \iota_C^*\sE \to \iota_C^*F^*_X\Omega_{X}^1 \to 0.
\end{equation}
Recall that $\iota_C^*F^*_X\Omega_X^1$ is ample and moreover, we have that $\iota_C^*\sE$ is not ample by assumption.

As $\iota_C^*\Omega_X^1$ is ample  \cite[Proposition 4.3]{Hartshorne:AmpleVectorBundle},  \autoref{lemma:ample_vb_curve} says that $\deg(\iota_C^*\Omega_X^1) > 0$. 
As the degree map on line bundles is integer-valued, we have that $\deg(\iota_C^*\Omega_X^1) \geq 1$. 
We need to do a calculation to ensure Tango's result (\autoref{thm:Tango_condition}) is applicable.  Give $\iota_C^*\Omega_X^1$ a decomposition 
\[
\iota_C^*\Omega_X^1 \cong \sF_1 \oplus \cdots \oplus \sF_n
\]
where $\sF_i$ are indecomposable vector bundles. 
As degree is additive with respect to direct sums, we have that 
\[
\deg(\iota_C^*\Omega_X^1) = \sum_{i=1}^n \deg(\sF_i).
\]
Since $\iota_C^*\Omega_X^1$ is ample, \cite[Proposition 2.2]{Hartshorne:AmpleVectorBundle} implies that each $\sF_i$ is ample, and hence $\deg(\sF_i)\geq 1$.  
Pulling back by an iterate of Frobenius will not change amplitude (as it is a finite surjective morphism \cite[\href{https://stacks.math.columbia.edu/tag/0CCD}{Tag 0CCD}]{stacks-project}). 
Using \cite[\href{https://stacks.math.columbia.edu/tag/01CI}{Tag 01CI}]{stacks-project}, \autoref{lemma:ample_vb_curve}, and \cite[Proposition 6.1.(b)]{Hartshorne:AmpleVectorBundle}, we have that 
\[
\deg(F_{C}^*\iota_C^*\Omega_X^1) = p \deg(\iota_C^*\Omega_X^1).
\]
Since $\deg(\sF_i) \geq 1$ and $p > 2\dim(X)^2g(C)$ where $g(C)$ is the genus of $C$, 
\[
\deg(F_{C}^*\sF_i) = p \deg(\sF_i) > 2\dim(X)^2g(C) > \dim(X)(\dim(X)-1) + (g-1) + \dim(X)\left( \frac{2g-2}{p}\right). 
\]
Therefore, \autoref{thm:Tango_condition} implies that for each $i = 1,\dots, n$
\[
H^1(C,\sF_i^{\vee}) \hookrightarrow H^1(C,F_{C}^*\sF_i^{\vee}) \hookrightarrow H^1(C,F_{C}^{2,*}\sF_i^{\vee}) \hookrightarrow \cdots 
\]
is an injection. 
Using properties of cohomology and direct sums, we have that each of the homomorphisms 
\[
H^1(C,\iota_C^*T_X) \hookrightarrow H^1(C,F_{C}^*\iota_C^*T_X) \hookrightarrow H^1(C,F_{C}^{2,*}\iota_C^*T_X) \hookrightarrow \cdots 
\]
is an injection. 

Returning to the above setup, since $\iota_C^*\sE$ is not ample, \autoref{coro:splitting_after_frob_curves} implies that pullback of the sequence \autoref{eqn:pullback_SES_J1_2} along suitable $n$-fold iterate of Frobenius will split. 
The claim about injectivity of cohomology along pullback along Frobenius and functoriality of Frobenii \cite[\href{https://stacks.math.columbia.edu/tag/0CC7}{Tag 0CC7}]{stacks-project} implies that the sequence \autoref{eqn:pullback_SES_J1_2} must in fact already be split! 
Now \autoref{lemma:pullback_split_J1_split} implies that the sequence 
\[
0\to \sO_{C} \to \sF \to F^*_C\Omega_{C}^1 \to 0
\]
corresponding to the class $\eta_{C} \in H^1(C,F^*_CT_{C})$ realizing the $\FT(C/k)$-torsor $J^1(C) \to C$ is also split. 
By \cite[Proposition 1.7]{Buium:GeometrypJets}, this splitting would force $C$ to have a lift of Frobenius. 
However, this contradicts a result of Raynaud \cite[I.5.4]{Ray83b} since \autoref{lemma:amplecotangent_subvarieties}.(2) implies that $g(C)\geq 2$. 

Therefore, we have shown that $\sE$ must be ample. 
To conclude the proof, recall that from Subsection \ref{subsub:specialfiber}, we have that $J^1(X)$ is identified with $\mP(\sE)\setminus \sO_{\mP(\sE)}(1)$. 
The ampleness of $\sE$ implies that $\sO_{\mP(\sE)}(1)$ is ample, and hence $J^1(X)$ is affine. 
The affineness of $J^n(X)$ for all $n\geq 1$ follows from the argument in \cite[Theorem 2.13]{Buium:GeometrypJets}. 
\end{proof}

\section{Quantitative unramified Manin-Mumford} 
\label{sec:explicitMM}
In this section, we prove \autoref{thm:mainthm1}. 
Before our proof, we recall the outline of Buium's argument from \cite{Buium:GeometrypJets}. In this work, Buium used his theory of arithmetic jets to prove an explicit upper bound on the intersection of a curve with torsion points in its Jacobian. 
Roughly, his proof consisted of three ingredients:
\begin{enumerate}
\item A deep theorem of Coleman \cite{Col87} concerning the ramification of torsion points lying a curve. 
\item An affineness result of the first arithmetic jet space associated to a hyperbolic curve, i.e.,  a curve of genus $g \geq 2$.  
\item Bounding the degree of certain subvarieties in the first arithmetic jet space of the curve's Jacobian relative to a chosen very ample line bundle. 
\end{enumerate}
Part (1) is based on Coleman's theory of abelian $p$-adic integrals on curves and detailed Newton polygon considerations. 
As such, it is not clear to what extent Coleman's result holds in the higher dimensional setting. 
We note that Buium used Coleman's result to reduce to the unramified local setting, which is how he deduced a complete Manin--Mumford statement.  
In \autoref{thm:main_Affine}, we proved a similar affineness statement to (2) for $W(\mF_p^{\alg})$-varieties whose special fiber has ample cotangent bundle and contain a smooth $W(\mF_p^{\alg})$-curve of low genus relative to the prime $p$.  

The goal of this section is consider part (3) in this higher dimensional setting. 
We remind the reader of our conventions from Subsection \ref{subsec:field_conventions} and \ref{subsec:0cycles}.

\begin{theorem}[= \autoref{thm:mainthm1}]
Let $X$ be a smooth projective $R$-subvariety of dimension $d$ of an abelian $R$-scheme $A$ of dimension $n$ such that $X_0$ has ample cotangent bundle, and suppose that there exists a smooth projective $R$-curve $C$ of genus $g \geq 2$ with a closed embedding $\iota\colon C \hookrightarrow X$ and that $p>2\dim(X)^2g$.  
Then,
\begin{equation}\label{eqn:body_MMbound}
\#(X(K) \cap A(K)[\tors]) \leq p^{3n} 3^{n} n!  \left( \sum_{i=0}^d  {2d \choose d + i} (-1)^is_i(F^{*}_{X_0}\Omega_{X_0}^1)\cdot \sO_{X_0}(3\Theta_0)^{d-i}\right)
\end{equation}
where $(-1)^is_i(F^{*}_{X_0}\Omega_{X_0}^1)$ is the $i$-th Segre class of $F^{*}_{X_0}\Omega_{X_0}^1$ and $\Theta_0$ is a theta divisor on the abelian variety $A_0$. 
\end{theorem}

Our proof largely follows the proof in \cite{Buium:GeometrypJets} with some modifications which we describe to remain self-contained. 

\subsection{Finiteness of unramified torsion cosets}
\label{subsec:finite_unramified}
We will use \autoref{thm:main_Affine}, our affineness result about $J^1(X_0)$,  to deduce finiteness of the left-hand side of \autoref{eqn:body_MMbound}. 
Consider the nabla map described in Subsection \ref{subsub:nabla} 
\[
\nabla_0 \colon A(R) \to J^1(A_0)(k).
\] 
Using the fact that the reduction map is torsion free \cite[Appendix]{Katz:GaloisProperties}, one can see that the restriction of $\nabla_0$ to the torsion subgroup $A(R)[\tors]$ of $A(R)$ is injective.

Recall that in \autoref{lemma:maximal_abelian_subvariety}, we identified $B\coloneqq p J^1(A_0)$ as the maximal abelian subvariety of $J^1(A_0)$. 
Let $D$ be the size of $\nabla_0(A(R)[\tors])$ under $J^1(A_0)(k) \to J^1(A_0)(k)/B(k)$, which is easily seen to be finite. 
Next, we note that 
\[
\nabla_0(X(R) \cap A(R)[\tors]) \subset \bigcup_{i = 1}^D (J^1 (X_0)(k) \cap  B_i(k))
\]
where $B_i(k) = b_i + B(k)$ is a translate of $B(k)$ by some $b_i \in J^1(A_0)(k)$. 
Moreover, as $\nabla_0$ is injective on $X(R) \cap A(R)[\tors]$, we have that 
\begin{equation}\label{eqn:bound1_body}
\# (X(R) \cap A(R)[\tors]) \leq \sum_{i=1}^D \# (J^1(X_0)(k) \cap B_i(k)). 
\end{equation}
Our assumptions and \autoref{thm:main_Affine} tell us that $J^1(X_0)$ is affine,  and as $B_i \coloneqq B + b_i$ is projective, we have that $J^1(X_0)(k) \cap B_i(k)$ is finite for each $i=1,\dots, D$. Therefore, we have that the intersection $(X(R) \cap A(R)[\tors])$ is finite. 

The quantitative approach then follows by finding bounds for $D$ and $\# (J^1(X_0)(k) \cap B_i(k))$. 
Bounding $D$ is straightforward. Buium \cite{Buium:GeometrypJets} showed that $D\leq p^{2n}$, and recently, Dogra--Pandit \cite[Lemma 3]{DograPandit} observed the improved estimate $D \leq p^{n}$. 

\subsection{Explicit bounds on unramified torsion cosets}
\label{subsec:explicitbound}
As evident from above, the bound from \autoref{eqn:body_MMbound} will come from explicitly bounding $\# (J^1(X_0)(k) \cap B_i(k))$. 
We will show that
\begin{equation}\label{eqn:bound2_body}
\# (J^1(X_0)(k)  \cap B_i(k)) \leq p^{2n} 3^{n} n!\left( \sum_{i=0}^{d} {2d \choose d+i} (-1)^is_i(F^{*}_{X_0}\Omega_{X_0}^1)\cdot \sO_{X_0}(3\Theta_0)^{d-i}\right)
\end{equation}
where we remind the reader of our conventions for $0$-cycles established in Subsection \ref{subsec:0cycles}. 

Similar to Buium \cite{Buium:GeometrypJets}, we will bound the left hand side of the \autoref{eqn:bound2_body} via an intersection theoretic computation.
To better understand these objects in the left hand side, we leverage the torsor structures of $J^1(X_0)$ and $J^1(A_0)$ and consider them as subschemes in the corresponding projective bundles as described in Subsection \ref{subsub:specialfiber}.

Each $J^1(A_0)$ and $J^1(X_0)$ are principal homogeneous spaces for the respective Frobenius tangent spaces, so we let $\sE_{A_0}$ and $\sE_{X_0}$ denote the vector bundles corresponding to each of these extension classes, respectively.  
This recognizes divisors $D_{X_0 } := \mP(F^*_{X_0 } \Omega_{X_0}^1)$ and $D_{A_0} := \mP(F^*_{A_0} \Omega_{A_0}^1)$ in $\mP(\sE_{X_0})$ and $\mP(\sE_{A_0})$ respectively, where we have the identifications $J^1(X_0) \cong \mP(\sE_{X_0})\setminus D_{X_0}$ and $J^1(A_0) \cong \mP(\sE_{A_0})\setminus D_{A_0}$. 
Letting $\iota\colon X_0 \hookrightarrow A_0$ denote the inclusion, we note that there is a surjective morphism $\iota^*\sE_{A_0} \to \sE_{X_0}$ extending the homomorphism $\iota^*F^*_{A_0}\Omega_{A_0}^1 \to F^*_{X_0}\Omega_{X_0}^1$, which in turn induces a closed embedding $\mP(\sE_{X_0}) \hookrightarrow \mP(\sE_{A_0})$. 
This will be compatible with the functorially induced map $D_{X_0} \to D_{A_0}$ by pullback along $\iota^*$, so it suffices to describe the map $J^1(X_0) \hookrightarrow J^1 (A_0)$ which can be done via cohomology. 

Specifically, setting $\eta_{X_0} \in H^1(X_0, F^*_{X_0} T_{X_0})$ and $\eta_{A_0} \in H^1(A_0, F^*_{A_0}T_{A_0})$ the classes corresponding to the extensions $0 \to \sO_{X_0} \to \sE_{X_0} \to  F^*_{X_0} \Omega_{X_0}^1 \to 0$ and $0 \to \sO_{A_0} \to \sE_{A_0} \to  F^*_{A_0} \Omega_{A_0}^1 \to 0$. Pulling back we have $0 \to \iota^* \sO_{A_0} \to \iota^* \sE_{A_0} \to \iota^*  F^*_{A_0} \Omega_{A_0}^1 \to 0$ and the natural map of short exact sequences induces a diagram 

\begin{center}
\begin{tikzpicture}
\node (TL) at (-4,0){$H^0(X_0, \iota^* \sO_{A_0})$};
\node (TC) at (0,0) {$H^1(X_0, \iota^*  F^*_{A_0} T_{A_0})$};
\node (TR) at (4,0){$H^1(X_0, \iota^* \sE_{A_0}^{\vee}).$};

\node (BL) at (-4,1){$H^0(X_0,  \sO_{X_0})$};
\node (BC) at (0,1) {$H^1(X_0,   F^*_{X_0} T_{X_0})$};
\node (BR) at (4,1){$H^1(X_0,  \sE_{X_0}^{\vee})$};

\draw[->] (TL) edge (TC);
\draw[->] (TC) edge (TR);

\draw[->] (BL) edge (BC);
\draw[->] (BC) edge (BR);

\draw[->] (BL) edge (TL);
\draw[->] (BC) edge (TC);
\draw[->] (BR) edge (TR);
\end{tikzpicture}
\end{center} 
The map $H^1(X_0,   F^*_{X_0} T_{X_0}) \to H^1(X_0, \iota^*  F^*_{A_0} T_{A_0})$ takes $\eta_{X_0}$ to the 
image of the class $\eta_{A_0}$, and this will induce the desired closed $k$-immersion $J^1(X_0) \hookrightarrow J^1(A_0)$.

Each projectivization has a projection map $\rho_{X_0} \colon \mP(\sE_{X_0}) \to X_0$ and $\rho_{A_0} \colon \mP(\sE_{A_0}) \to A_0$. 
Consider the line bundle 
\[
\sH := \rho_{A_0}^* \sO_{A_0}(3 \Theta_0) \otimes \sO_{\mP(\sE_{A_0})}(1)
\] 
on $\mP(\sE_{A_0})$ where $\Theta_0$ is a $\Theta$-divisor on $A_0$. 
Using the argument from \cite[p.~4]{BV96} and \cite[p.~163, Theorem]{MumAb}, we see that $\sH$ is very ample.  
With this very ample line bundle, we obtain a closed embedding with respect to which we can calculate degree. 
In particular, we may use B\'ezout's theorem \cite[p.~148]{Fulton:IntersectionTheory} to say that 
\begin{equation}\label{eqn:bound3_body}
\# (J^1(X_0)(k)  \cap B_i(k)) \leq (\deg_ {\sH}B_i)\cdot (\deg_{\sH}\mP(\sE_{X_0})).
\end{equation}
Note that the argument bounding $\deg_{\sH} B_i$ on \cite[p.~357]{Buium:GeometrypJets} goes through \textit{mutatis mutandis} in our setting, so we have 
\begin{equation}\label{eqn:bound4_body}
\deg_{\sH} B_i \leq  p^{2n} 3^n  n!
\end{equation}
where $\dim_k A_0 = n$. 
Furthermore, we see that it suffices to bound $\deg_{\sH}\mP(\sE_{X_0})$.

\begin{prop}\label{prop:intersection_ProjEX}
With the notation as above, we have that 
\[
\deg_{\sH}\mP(\sE_{X_0}) = \sum_{i=0}^{d} {2d \choose d+i} (-1)^is_i(F^{*}_{X_0}\Omega_{X_0}^1)\cdot \sO_{X_0}(3\Theta_0)^{d-i}
\]
where $(-1)^is_i(F^{*}_{X_0}\Omega_{X_0}^1)$ is the $i$-th Segre class of $F^*_{X_0}\Omega_{X_0}^1$ and the product refers to the intersection product on the Chow ring of $X_0$. 
\end{prop}

\begin{proof}
First, we note that 
\[
\sH\otimes \sO_{\mathbf{P}(\sE_{X_0})} \cong \rho_{X_0}^*\sO_{X_0}(3\Theta_0)\otimes \sO_{\mathbf{P}(\sE_{X_0})}(1).
\]
As $\sH$ is very ample, it suffices to compute the self-intersection number
\[
(\sH\otimes \sO_{\mathbf{P}(\sE_{X_0})})^{2d} = (\rho_{X_0}^*\sO_{X_0}(3\Theta_0)\otimes \sO_{\mathbf{P}(\sE_{X_0})}(1))^{2d}
\]
where we note that $\dim_k \mP(\sE_{X_0}) = 2d$. 
We begin by simply expanding the above expression using basic properties of intersection products 
\[
(\rho_{X_0}^*\sO_{X_0}(3\Theta_0)\otimes \sO_{\mathbf{P}(\sE_{X_0})}(1))^{2d} = \sum_{i=0}^{2d}{2d \choose i} (\rho_{X_0}^*\sO_{X_0}(3\Theta_0)^{2d-i}\cdot \sO_{\mathbf{P}(\sE_{X_0})}(1)^i )
\]
We will analyze this summation by breaking it into two parts. 

To begin, we first claim for $0\leq i \leq d-1$, we have that 
\[
 (\rho_{X_0}^*\sO_{X_0}(3\Theta_0)^{2d-i}\cdot \sO_{\mathbf{P}(\sE_X)}(1)^i ) = 0.
\]
Note that 
$
\rho_{X_0}^*\sO_{X_0}(3\Theta_0)^{2d-i} = \rho_{X_0}^*(\sO_{X_0}(3\Theta_0)^{2d-i})
$
where the intersection on the right-hand side takes place in $X_0$. 
As $\dim X_0 = d$, we have that $\sO_{X_0}(3\Theta_0)^{2d-i} = \emptyset$ for $0\leq i \leq d-1$ which implies our claim. 
Thus, it suffices to compute 
\[
\sum_{i=d}^{2d}{2d \choose i} (\rho_{X_0}^*\sO_{X_0}(3\Theta_0)^{2d-i}\cdot \sO_{\mathbf{P}(\sE_{X_0})}(1)^i )
\]

Since the above intersections in this summation are all 0-cycles, we may and do compute their intersection numbers after pushforward along $\rho_{X_0}$. 
By the definition of the $i$-th Segre classes from Subsection \ref{subsub:Segreclasses}, we can rewrite the summation as
\[
\sum_{i=d}^{2d}{2d \choose i} (\sO_{X_0}(3\Theta_0)^{2d-i}\cdot (-1)^{i-d}s_{i-d}(\sE_{X_0}) ),
\]
After re-indexing, the summation has the following shape
\[
\sum_{i=0}^{d}{2d \choose d + i} (\sO_{X_0}(3\Theta_0)^{d-i}\cdot (-1)^is_{i}(\sE_{X_0})  ).
\]
To finish our claim, we note that $\sE_{X_0}$ sits in the short exact sequence
\[
0\to \sO_{X_0} \to \sE_{X_0} \to F^*_{X_0}\Omega_{X_0}^1 \to 0
\]
The multiplicativity of Segre classes (\autoref{lemma:SegreWhitneySum}) implies that $s(\sE_{X_0}) = s(\sO_{X_0})\cdot s(F^{*}_{X_0}\Omega_{X_0}^1)$ but since $s(\sO_{X_0}) = 1$, we have that $s(\sE_{X_0}) = s(F^{*}_{X_0}\Omega_{X_0}^1)$, and so our claim follows. 
\end{proof}

We now conclude with a proof of \autoref{thm:mainthm1}.

\begin{proof}[Proof of \autoref{thm:mainthm1}]
The bound from \autoref{eqn:bound2_body} follows from \autoref{eqn:bound3_body}, \autoref{eqn:bound4_body}, and \autoref{prop:intersection_ProjEX}.  
The final result follows from combining \autoref{eqn:bound1_body} with  \autoref{eqn:bound2_body} and recalling that $\# D \leq p^n$ from \cite[Lemma~3]{DograPandit}. 
\end{proof}

\begin{remark}\label{rem:Buium_bound}
When $X$ is a curve of genus $g\geq 2$ and $A$ is the curve's Jacobian, the bound from \autoref{prop:intersection_ProjEX} recovers Buium's self intersection computation from \cite{Buium:GeometrypJets}. 
Indeed, Buium showed that 
\[
\deg_{\sH}\mP(\sE_{X_0}) = 6g + p(2g-2).
\]
The bound from \autoref{prop:intersection_ProjEX} is 
\[
\deg_{\sH}\mP(\sE_{X_0}) = 2\sO_{X_0}(3\Theta_0)\cdot s_0(F^{*}_{X_0}\Omega_{X_0}^1) + (-1)s_1(F^{*}_{X_0}\Omega_{X_0}^1). 
\]
Standard properties of theta divisors $\Theta_0$ imply that
\begin{align*}
2\sO_{X_0}(3\Theta_0)\cdot s_0(F^{*}_{X_0}\Omega_{X_0}^1) &= 2(\sO_{X_0}(3\Theta_0) \cdot X_0 ) = 2\cdot 3 \cdot (\deg_{\Theta_0}X_0) = 2\cdot 3 \cdot g.
\end{align*}
Now, \autoref{lemma:SegreChernInverse} implies that 
\[
(-1)s_1(F^{*}_{X_0}\Omega_{X_0}^1) = c_1(F^{*}_{X_0}\Omega_{X_0}^1) = \deg(F^{*}_{X_0}\Omega_{X_0}^1) = p(2g-2),
\]
and therefore,  we have our desired equality.  
\end{remark}

\section{Proof of Manin--Mumford for varieties with ample cotangent bundle}
\label{sec:MMProof}
In this section, we prove \autoref{thm:explicit_unramified_MM} and illustrate how we can use this to deduce \autoref{thm:general_MM}. 
For the reader's convenience, we recall these theorems below. 

\begin{theorem}[= \autoref{thm:explicit_unramified_MM}]\label{thm:explicit_unramified_MM_indoc}
Let $F$ be a number field, $A$ an abelian $F$-variety of dimension $n$, and $X$ a smooth $F$-subvariety of dimension $d$ in $A$ such that the cotangent bundle $\Omega_X^1$ is ample. For every $p\gg 0$ of good reduction for both $X$ and $A$, we have that 
\[
\# (X(F^{\alg}) \cap A(F^{\alg})[\nptors]) \leq p^{3n} 3^{n} n!  \left( \sum_{i=0}^d  {2d \choose d + i} (-1)^is_i(F^{*}_{X_0}\Omega_{X_0}^1)\cdot \sO_{X_0}(3\Theta_0)^{d-i}\right)
\]
where 
$A(F^{\alg})[\nptors]$ denotes the prime-to-$p$ torsion in $A(F^{\alg})$, $X_0$ and $A_0$ denote the reduction of $X$ and $A$ modulo $p$, respectively, $(-1)^is_i(F^{*}_{X_0}\Omega_{X_0}^1)$ is the $i$-th Segre class of the Frobenius pullback of the cotangent bundle of $X_0$, and $\Theta_0$ is a $\Theta$-divisor on the abelian variety $A_0$. 
\end{theorem}

\begin{theorem}[= \autoref{thm:general_MM}]\label{thm:MM_indoc}
Let $F$ be a number field, let $A$ be an abelian $F$-variety, and let $X$ be a smooth $F$-subvariety of $A$. If $\Omega_X^1$ is ample,  then
\[
\# (X(F^{\alg}) \cap A(F^{\alg})[\tors]) < \infty. 
\]
\end{theorem}

\subsection{Manin--Mumford for prime-to-$p$ torsion cosets}
First, we discuss the proof of \autoref{thm:explicit_unramified_MM_indoc}. 
We remind the reader of our conventions from Subsection \ref{subsec:Conventions}. 

The first step in our proof is to reduce to the local setting. 
This follows from the well-known fact (see e.g., \cite[Theorem 4]{Roessler:NoteMM}) that for $p$ a prime of good reduction for $A$, 
\begin{equation}\label{eqn:equal_unramified_torsion}
A(F^{\alg})[\nptors] = A(K)[\nptors]
\end{equation}
where $K$ is the maximal extension of $F$ contained in $F^{\alg}$ that is unramified above $p$.   
In \autoref{thm:mainthm1}, we proved a finiteness result for $\# (X(K) \cap A(K)[\nptors])$ for $W(\mF_p^{\alg})$-varieties whose special fiber has ample cotangent bundle and that contains a $W(\mF_p^{\alg})$-curve of good reduction with low genus relative to $p$. The next lemma will show how we may use \autoref{thm:mainthm1} to deduce \autoref{thm:explicit_unramified_MM}. 

\begin{lemma}\label{lemma:findingprime}
Let $F$ be a number field, and let $X$ be a smooth $F$-subvariety of an abelian $F$-variety $A$ such that $\Omega_X^1$ is ample.  
There exists a prime number $p$ such that there exists
\begin{enumerate}
\item a prime $\mathfrak{p}$ of $F$ over $p$ such that $F/\mQ$ is unramified at $\mathfrak{p}$, 
\item smooth projective $W(\mF_p^{\alg})$-varieties $\mathfrak{X}_p$ and $\mathfrak{A}_p$ and a smooth projective $W(\mF_p^{\alg})$-curve $\mathfrak{C}_p$ such that 
\begin{enumerate}
\item there exist closed $W(\mF_p^{\alg})$-immersions $\mathfrak{C}_p \hookrightarrow \mathfrak{X}_p \hookrightarrow \mathfrak{A}_p$,  
\item the generic fiber of $\mathfrak{X}_p$ (resp.~$\mathfrak{A}_p$) of isomorphic to the base change of $X$ (resp.~$A$) to $\Spec(W(\mF_p^{\alg})[1/p])$, 
\item the genus $g$ of the special fiber $\mathfrak{C}_{p,0}$ satisfies the inequality $p > 2\dim(X)^2 g$, and 
\item the special fiber $\mathfrak{X}_{p,0}$ of $\mathfrak{X}_p$ has ample cotangent bundle. 
\end{enumerate}
\end{enumerate}
\end{lemma}

\begin{proof}
First, we may iteratively use Bertini's theorem \cite[Cor.I.6.11(2)]{Jouanolou:Bertini} to prove that there exists infinitely many smooth curves $C$ inside of $X$. 
Since smoothness is an open condition (i.e., there are only finitely many places of bad reduction for smooth projective $F$-varieties),  there exists some finite set of places $S_1$ of $F$ such that $C$ and $X$ can be spread out to smooth projective varieties $\mathfrak{C}$ and $\mathfrak{X}$ over $\Spec(\sO_{F,S_1})$ and $A$ can be spread out to an abelian scheme $\mathfrak{A}$ over $\Spec(\sO_{F,S_1})$. 
Let $\pi\colon \mathfrak{X} \to \Spec(\sO_{F,S_1})$ denote the structure morphism, which is smooth and projective. 
As $\pi$ is smooth, the relative cotangent bundle $\Omega_{\mathfrak{X}/\Spec(\sO_{F,S_1})}^1$ is a vector bundle. 
We have that the fiber of $\pi$ over the generic point $\eta$ of $ \Spec(\sO_{F,S_1})$ is isomorphic to $X$, which was assumed to have ample cotangent bundle. 
The fiberwise criterion for relative ampleness of vector bundles \cite[Proposition 4.4]{Hartshorne:AmpleVectorBundle} tells us that there exists some open neighborhood $U$ of $\eta$ such that for each point $x\in U$, the restriction of $\Omega_{\mathfrak{X}/\Spec(\sO_{F,S_1})}^1$ to the fiber $\pi^{-1}(x)$, which we realize as $\Omega_{\mathfrak{X}_x/\Spec(\kappa(x))}^1$, is ample.  
Using the above and taking $p \gg 0$, we can base change the smooth projective $\sO_{F,S_1}$-varieties to $W(\mF_p^{\alg})$ and guarantee that conditions $(2).(a,b,c,d)$ are satisfied. 
We may also guarantee conditions $(1)$ and $(2).(a,b,c,d)$ hold simultaneously by noting that there exist infinitely many unramified primes in $F$, which concludes the proof. 
\end{proof}

We may now prove \autoref{thm:explicit_unramified_MM}. 

\begin{proof}[Proof of \autoref{thm:explicit_unramified_MM}]
This follows from the \autoref{eqn:equal_unramified_torsion}, \autoref{lemma:findingprime}, and \autoref{thm:mainthm1}.
\end{proof}

\subsection{Proof of \autoref{thm:general_MM}}
For the proof of our general Manin--Mumford statement (\autoref{thm:MM_indoc}), we begin by noting that standard techniques (see e.g.,~\cite{Raynaud:Torsion,Roessler:NoteMM,  PinkRoessler:HrushovskiMM,Ito:MMSupersingular}) show that proving a general Manin--Mumford statement reduces to proving a $p$-power and prime-to-$p$ Manin--Mumford statement i.e.,  to prove that $\# (X(F^{\alg}) \cap A(F^{\alg})[\tors]) < \infty$, it suffice to prove following
\begin{align*}
\# (X(F^{\alg}) \cap A(F^{\alg})[p^{\infty}]) &< \infty \\
\# (X(F^{\alg}) \cap A(F^{\alg})[\nptors]) &< \infty
\end{align*}
for some prime $p$. 
By \autoref{thm:explicit_unramified_MM}, there exists a prime $p\gg 0$ such that $$\# (X(F^{\alg}) \cap A(F^{\alg})[\nptors]) < \infty,$$ so it suffices to bound the intersection of our variety with the $p$-power torsion of the abelian variety.  
To do so, we use a result of Bogomolov. 

\begin{theorem}[Bogomolov +\,$\varepsilon$]\label{thm:Bogomolov}
Let $F$ be a number field, let $A$ be an abelian $F$-variety, and let $X$ be a smooth $F$-subvariety of $A$ such that $\Omega_X^1$ is ample.  
For any prime $p\geq 2$, we have that 
\[
\# (X(F^{\alg}) \cap A(F^{\alg})[p^{\infty}]) < \infty. 
\]
\end{theorem}

\begin{proof}
This follows immediately from \cite[Th\'eor\`eme 3]{Bogomolov:LAdicReps} and \autoref{lemma:amplecotangent_subvarieties}.(1).  
\end{proof}

By combining the above discussion, \autoref{thm:Bogomolov}, and \autoref{thm:explicit_unramified_MM}, we have a proof of \autoref{thm:general_MM}. 

\section{Explicit Manin--Mumford for complete intersections of general hypersurfaces}
\label{sec:computation}
In this section, we illustrate how to apply our main theorems to the setting of complete intersections in an abelian variety. 
Additionally, we show that the bound from \autoref{thm:explicit_unramified_MM} only depends on certain parameters related to the complete intersection. More precisely, the goal of this section is to prove the following.

\begin{theorem}[= \autoref{thm:intro_complete_intersection_bound}]\label{thm:complete_intersection_bound}
Let $A$ be an abelian $\mQ$-variety of dimension $n$, and $X$ a smooth $\mQ$-subvariety of $A$ that is $\mQ^{\alg}$-isomorphic to the intersection of $c >n/2$ sufficiently ample general hypersurfaces $H_1,\dots, H_c $ of large and divisible enough degrees $d_1,\dots,d_c$ in $A_{\mQ^{\alg}}$. 
For every sufficiently large prime $p$, there exists some constant $\alpha(p,n,\underline{I})$ that is polynomial in $p,n$, and intersection numbers $\underline{I}$ of certain products of the hypersurfaces such that
\[
\# (X(\mQ^{\alg}) \cap A(\mQ^{\alg})[\nptors]) \leq \alpha(p,n,\underline{I}).
\]
\end{theorem}

First, we recall of Debarre \cite{Debarre:VarietiesAmpleCotangent, Debarre:CorrigendumVarietiesAmple} concerning ampleness of the cotangent bundle of general complete intersections. 

\begin{theorem}[\protect{\cite[Theorem 8]{Debarre:VarietiesAmpleCotangent, Debarre:CorrigendumVarietiesAmple} }]\label{thm:Debarre}
Fix $K$ an algebraically closed field of characteristic zero, and let $A$ be an abelian $K$-variety of dimension $n$. 
For $c > n/2$, let $L_1,\dots,L_c$ be very ample line bundles on $A$, let $e_2,\dots,e_c$ be large and divisible enough positive integers, and let $H_1 \in |L_1^{e_1}|,\dots,H_c\in |L_c^{e_c}|$ be general divisors. 
The cotangent bundle $\Omega_V^1$ of $V = H_1\cap \cdots \cap H_c$ is ample. 
\end{theorem}

In \cite[Theorem 7]{Debarre:VarietiesAmpleCotangent}, Debarre shows that if $A$ is assumed to be simple, then one may take $e_2,\dots,e_c > n$. 
We refer the reader to \textit{loc.~cit.}~for details on what the terms large and divisible enough refer to in \autoref{thm:Debarre}. 
We are now ready to prove \autoref{thm:complete_intersection_bound}.

\begin{proof}[Proof of \autoref{thm:complete_intersection_bound}]
By our assumptions on $X$, \autoref{thm:Debarre} implies that $X_{\mQ^{\alg}}$ has ample cotangent bundle, and via \cite[\href{https://stacks.math.columbia.edu/tag/0D2P}{Tag 0D2P}]{stacks-project}, we have that $X$ has ample cotangent bundle as well. 
Let $p$ be a sufficiently large prime such that $A$ has good reduction at $p$, satisfying \autoref{lemma:findingprime}, and so that $X$ remains a complete intersection modulo $p$.
The assumptions on $p$ tells us that the conditions of \autoref{thm:mainthm1} are satisfied, and hence 
\begin{equation}\label{eqn:proof_complete_int_bound}
\#(X(\mQ^{\alg}) \cap A(\mQ^{\alg})[\nptors]) \leq p^{3n} 3^{n} n!  \left( \sum_{i=0}^{d}  {2d \choose d + i} (-1)^is_i(F^{*}_{X_{\mF_{p}^{\alg}}}\Omega_{X_{\mF_{p}^{\alg}}}^1)\cdot \sO_{X_{\mF_{p}^{\alg}}}(3\Theta_0)^{d-i}\right).
\end{equation}
where $d = \dim(X)$.

It remains to deduce that the right-hand side of the above equation can expressed as a polynomial in the parameters $p,n$, and intersection numbers $\underline{I}$ of certain products of the hypersurfaces.  
We note that the part of \autoref{eqn:proof_complete_int_bound} in the parentheses is exactly the same as \cite[Equation (8)]{Scarponi:Sparsity}, after tracing through different notations and reindexing terms.  
Moreover in \textit{loc.cit.~}Remark 6.2 the author computes an expression for this part of \autoref{eqn:proof_complete_int_bound} which is polynomial in the above parameters. 
Our claim concerning the structure of $\alpha$ follows. 
\end{proof}

  \bibliography{refs}{}
\bibliographystyle{amsalpha}

\end{document}